\documentclass{amsart}
\usepackage{amsmath,amsfonts,amssymb,amsthm}
\usepackage{hyperref}
\usepackage{mathrsfs}
\usepackage[all]{xy}
\usepackage{epsfig}
\usepackage{mathrsfs}	
\usepackage{bm}
\usepackage[english]{babel}
\usepackage[latin1]{inputenc}

\newtheorem{theorem}{Theorem}[section]

\newtheorem{proposition}[theorem]{Proposition}
\newtheorem{lemma}[theorem]{Lemma}
\newtheorem{example}[theorem]{Example}

\newtheorem{definition}[theorem]{Definition}
\newtheorem*{question*}{Question} 
\newtheorem*{problem*}{Problem} 
\newtheorem{algorithm}[theorem]{Algorithm}
\newtheorem{remark}[theorem]{Remark}

\newtheorem{hyp}[theorem]{Hypotheses}

\numberwithin{equation}{section}

\newcommand\paren[1]{\left(#1\right)}
\newcommand\llave[1]{\left\{#1\right\}}
\DeclareMathOperator\vol{vol}

\newcommand\dto{\dashrightarrow}

\newcommand\lto{\longrightarrow}
\newcommand\nto{\stackrel}

\def\NN{\mathbb{N}}
\def\ZZ{\mathbb{Z}}
\def\kk{\mathbb{K}}
\def\PP{\mathbb{P}}
\def\QQ{\mathbb{Q}}
\def\CC{\mathbb{C}}
\def\RR{\mathbb{R}}
\def\AA{\mathbb{A}}

\def\Cc{\mathscr{C}}

\def\Sc{\mathscr{S}}

\def\Tc{\mathscr{T}}

\def\Nc{\mathcal{N}}

\def\Ic{\mathcal I}
\def\Z.{\mathcal{Z}_\bullet}
\def\M.{\mathcal{M}_\bullet}

\def\mm{\mathfrak{m}}

\def\Fitt0{\mathfrak{F}}

\def\deg{\mathrm{deg}}

\def\im{\mathrm{im}}

\def\Res{\mathrm{Res}}

\def\Syz{\mathrm{Syz}}

\def\dim{\mathrm{dim}}

\newcommand\SIR{\mathrm{Sym}_A (I)}
\newcommand\RIR{\mathrm{Rees}_A (I)}

\title[Implicitization via syzygies]
{Implicitization of rational hypersurfaces via linear syzygies: a practical overview}

\author{Nicol\'as Botbol and Alicia Dickenstein}
\address{Departamento de Matem\'atica,
FCEN, Univ. de Buenos Aires, and IMAS-CONICET, Buenos Aires, Argentina.}
\email{{alidick,nbotbol}@dm.uba.ar}
\thanks{Partially supported by UBACYT 20020100100242, CONICET PIP 112-200801-00483 
and ANPCyT PICT 2013-1110, Argentina}
\date{\today}

\begin{document}

\begin{abstract}
We unveil in concrete terms the general machinery of the 
syzygy-based algorithms for the implicitization of
rational surfaces in terms of the monomials in the 
polynomials defining the parametrization, following and
expanding our joint article with M. Dohm.
These algebraic techniques, based on the theory of
approximation complexes due to J. Herzog, A, Simis and W.
Vasconcelos, were  introduced for the implicitization
problem by J.-P. Jouanolou, L. Bus\'e, and M. Chardin. Their work was inspired by the
practical method of moving curves, proposed by T. Sederberg and
F. Chen, translated into the language of syzygies by D. Cox. Our aim is to express
the theoretical results and resulting algorithms into very concrete terms,
avoiding the use of the advanced homological commutative algebra tools which are needed
for their proofs. 
\end{abstract}

\maketitle


\section{Introduction}\label{sec:intro}

Let $\kk$ be a field. We can assume $\kk=\QQ$ (or any computable field) 
when dealing with implementations. 
All the varieties, rings and vector spaces we will consider are 
understood to be taken over $\kk$. Consider a rational parametrization
\begin{eqnarray}\label{eq:f}
 \kk^2 & \stackrel{f}{\dashrightarrow} & \kk^3 \notag \\
s =(s_1,s_2) & \mapsto & \left(\frac{f_1(s)}{f_0(s)},
\frac{f_2(s)}{f_0(s)},\frac{f_3(s)}{f_0(s)}\right)
\end{eqnarray}
of a (hyper)surface $\Sc := (F=0)\subset \AA^3$, where $F \in\kk[T_1,T_2,T_3]$ 
is a non-constant polynomial and $f_i \in \kk[s_1,s_2]$. 
(As usual, the dashed arrow means that $f$ is defined on a dense open set of $\kk^2$.)
 An important problem in
 computer aided geometric design is to  switch from parametric to implicit 
representations of rational surfaces \cite{ho89}, that is the \emph{parametrization} $f$ is 
assumed to be known and one seeks for 
the 
\emph{implicit equation} $F$ (which is defined  only up to multiplicative
constant).
In fact, we will assume that $f$ is given and 
our aim will not be to get the implicit equation $F$ of $\Sc$ written in terms of 
its monomials, but a 
\emph{matrix representation} of the surface.

\begin{definition}\label{def:MatRep}
A matrix representation $M$ of $\Sc$ is a matrix with entries in $\kk[T_1,T_2,T_3]$, 
generically of full rank, which verifies the following condition: for any point $p \in \kk^3$, 
the rank of $M(p)$ drops if and only if $p$ lies on $\Sc$.
\end{definition}

The use of matrix representations goes back to Manocha and Canny~\cite{MC91}, and to
Chionh and Goldman~\cite{CG92}. Having the matrix $M$ is 
sufficiently good for many 
purposes and it is cheaper to compute. The well-developed theory and tools of linear algebra can be applied 
to solve geometric problems. We can certainly use the (numerical) 
rank dropping condition in Definition~\ref{def:MatRep}
to check membership in $\Sc$, and, moreover, the whole structure of minors of $M$ is related to the
singularities of the parametrization \cite{BBC13} and gives a way to invert it when the fiber has
a single point  \cite{buse13,BBC13}.
Matrix representations are also useful for solving intersection problems
as is shown in \cite{ACGS07,BT09,DFGS13,buse13}. Much of the computational
difficulty in these problems lies on computing ranks for polynomial
matrices (cf.\ \cite{HS98} as well as Section~5 in the nice and interesting paper~\cite{buse13}).

\medskip

The motivation for this paper is to present in the simplest possible 
terms procedures for the implicitization of
rational surfaces via matrix representations, based on the syzygies $(h_0, \dots, h_3)$ of the 
input polynomials, that is, 4-tuples of polynomials
in the $s$ variables verifying the linear relation $\sum_{i=0}^3 h_i f_i=0$. 
The theoretical justification is not naive and 
requires a good command of techniques of (homological) commutative algebra. 
However, the algorithms do not require a heavy background and are
easy to explain. We will show that they perform very well, and moreover, 
they work even better in the presence
of base points. 

Call $T_1, T_2, T_3$ the coordinates in the target of $f$.
Our question is an instance of elimination of variables, where we want to 
find the algebraic relations among the variables $T_1, T_2, T_3$ under the
assumption that $f_0(s) T_i- f_i(s) =0, i = 1,2,3$, for some $s$ in the domain of $f$.
The eliminant polynomial by excellence is the determinant $\det(A)$, 
a polynomial with integer coordinates on the coefficients of a 
square matrix $A$, which vanishes on those coefficients for which 
there exists a nonzero solution $x$ to the equations
$A \cdot x =0$.  Elimination of variables is done in the literature through different 
incarnations of the following general strategy:
\begin{enumerate}
\item Reduce the problem to a linear algebra problem.
\item Hide the variables one wants to eliminate in the (typically monomial) bases.
\item Use determinants.
\end{enumerate}
This strategy is also the core in our syzygy-based algorithms.

The following short account of the approach of the use of syzygies
in our context is reconstructed from David Cox's lecture 
at the Conference PASI on Commutative Algebra and its connections to 
Geometry honoring Wolmer Vasconcelos, held in Brazil in 2009 \cite[Mini Course 1]{Olinda09}.
The use of syzygies for the implicitization of (conic) surfaces goes 
back to Steiner in 1832 \cite{St32}.
In 1887, Meyer describes in \cite{Me87} syzygies of three 
polynomials and makes a general conjecture proved by Hilbert in 1890 \cite{Hi90}.
Surface implicitization by eliminating parameters was studied by 
Salmon in 1862 \cite{Sa62} and Dixon in 1908 using resultants \cite{Di08}.
In 1995, Sederberg and Chen reintroduced the use of syzygies, 
by a method termed as  \emph{Moving curves and surfaces} \cite{SC95}.
Cox realized they were using syzygies \cite{Co01}, and produced several papers with other coauthors 
(Bus\'e, Chen, D'Andrea, Goldman, Sederberg, Zhang \cite{BCD03,Co03,CSC98,ZSCC03}).
In 2002, Jouanolou and Bus\'e  \cite{BuJo03} abstracted and generalized 
on a sound basis the method of Sederberg-Chen 
via approximation complexes, a tool in homological commutative 
algebra that had been developed by Herzog, Simis and Vasconcelos \cite{HSV,HSV1,HSV2}.
Bus\'e, Chardin, Jouanolou and Simis produced further advances 
in the homogeneous case \cite{BuJo03, BC05, Ch06, BCJ06,BCS10}. 
Goldman et al. studied the cases of planar and space curves \cite{JG09,HWJG10,JWG10}.
A generalization of the 
linear syzygy method when the support of the input polynomials 
is a square (that is, bihomogeneous of degree $(d,d)$) was proposed by
Bus\'e and Dohm \cite{BD07}, and for any polygon by Botbol, 
Dickenstein and Dohm \cite{BDD08}, and Botbol \cite{Bot08, Bot09, Bot10}.
This method is particularly adapted when the polynomials 
defining the parametrization are \emph{sparse}, which is often the case.
This will be our point of view in this article.
So, we want to solve the following problem.

\begin{problem*}\label{pb:main}
 Given a rational parametrization
$f$ as in\eqref{eq:f}, find a matrix representation $M$ of the surface $\Sc$
by means of syzygies and the monomial structure of $f_0, \dots, f_3$.
\end{problem*}

The main general algorithmic answer to this problem is given in Algorithm~\ref{algo:mainalgo} 
(see Theorem~\ref{th:mainalgo}).
Our assumption that the dimension of $\Sc$ is $2$ is equivalent to the fact, 
when we extend the map to the algebraic closure
$\overline{\kk}$ of $\kk$, that for almost all $p=f(s)$ in the image of $f$, 
the number of preimages by $f$ is finite. 
This number is called the degree of $f$ and noted $\deg(f)$.
The matrix representations $M$ of $\Sc$ provided by Algorithms~\ref{algo:mainalgo} and~\ref{algo:Hirzebruch}
moreover satisfy  that the greatest common divisor of all minors of 
$M$ of maximal size equals $F^{\deg(f)}$. 

\medskip

We present in Section~\ref{sec:naive} the first naive linear algebra 
algorithm to compute the implicit equation $F$, which requires to solve a huge linear algebra 
system. Moreover, this naive method ``forgets'' the parametrization and thus in general it is
not useful in Computer Aided Geometric Design.
In Section~\ref{sec:syzygies} we recall previous results on the
implicitization of curves and surfaces using syzygies and present 
our general methods of implicitization via linear syzygies, 
which requires to solve considerably smaller systems. 
We highlight in Section~\ref{ss:curves}
the main elimination step, which was termed \emph{instant elimination}
in \cite{EOWR} (see also the references therein). 

In Section~\ref{ssec:bihom} we present 
in Theorem~\ref{th:Hirzebruch} a refinement of Theorem~\ref{th:mainalgo} 
for bihomogeneous parametrizations, in the same spirit. 
Technicalities are avoided 
in our presentation in these sections, and in particular in the statement 
of our main results Theorems~\ref{th:mainalgo} and~\ref{th:Hirzebruch}.

Detailed hypotheses and proofs are deferred to
Section~\ref{sec:tgtools}, where we introduce the necessary background on
toric geometry. We collect in Appendix~\ref{sec:catools} 
a general overview of the rationale of the  tools and results from homological 
commutative algebra required for the proofs. A reader only interested in
the application of our results, can skip these two sections.

Section~\ref{sec:examples} illustrates the practicality and advantages
of our approach. For our computations, we use implementations in Macaulay 2,
which need different type of homogenizations to
use current routines (via a toric embedding or a multihomogenization via an abstract toric Cox ring) 
\cite{BD10,Bot10M2}.\footnote{Routine updates at:
{\tt http://mate.dm.uba.ar/\~{}nbotbol/Macaulay2/BigradedImplicit.m2}, {\tt 
http://mate.dm.uba.ar/\~{}nbotbol/Macaulay2/MatrixRepToric.m2}.}  
For the best performance of our algorithms, it would be important 
to design ad-hoc structured linear algebra strategies  
to compute syzygies in the sparse case.

\section{A naive linear algebra answer}\label{sec:naive}

The convex hull in $\RR^n$ of the exponents of the monomials 
occurring in a non zero (Laurent) 
polynomial $h$ in $n$ variables is called the Newton polytope 
$\Nc(h)$ of $h$. When $h$ is a 
polynomial in $(s_1,s_2)$ of degree (at most) $d$, its Newton 
polygon $\Nc(F)$ is contained in 
the triangle $\Delta_d$ with vertices $(0,0), (d,0), (0,d)$. 
The Euclidean area $\vol(\Delta_d)$ of this triangle is $d^2/2$ 
and its lattice area $\vol_\ZZ(\Delta_d)$
is equal to $2 \cdot d^2/2 = d^2$, which is always an integer.

We have the following classical result (c.f. for instance \cite{BuJo03}):

\begin{theorem}
 For generic polynomials $f_0,\dots,f_3$ of degree $d$, 
 the degree of the implicit equation $F$ is 
 $d^2$ and its Newton polytope is the tetrahedron with vertices 
 $(0,0,0), (d^2,0,0), (0,d^2,0), (0,0,d^2)$.
\end{theorem}

In the sparse case, the following generalization holds \cite{SY94}.

\begin{theorem} 
For generic  polynomials $f_0, \dots, f_3$ with the same Newton polygon $P$, 
the degree of $F$ is the lattice area $v=\vol_\ZZ(P)$  and its Newton polytope 
is the tetrahedron with vertices $(0,0,0), (v,0,0), (0,v,0)$, $(0,0,v)$. 
\end{theorem}

A first naive algorithm would then be the following.
Assume the Newton polytope $\Nc(F)$ of $F$ is known (as in the previous theorems) 
and number $m_1, \dots, m_N \in \NN^3$ the integer points (also called lattice points) 
in $\Nc(F)$.  Consider indeterminates $c =(c_1, \dots, c_N)$ and write 
$F=\sum_{i=1}^N c_i T^{m_i}$. Substitute $T = f(s) $ and equate to zero 
the coefficient of each power of $(s_1,s_2)$ that occurs (clearing the denominator). 
This sets a system $\mathcal L$ of linear equations in $c$, with solution space 
of dimension $1$. Any nonzero solution $c$ of $\mathcal L$ will give a choice of 
implicit equation $F$.

\medskip

This solves the problem, but, which is the size of this linear system  $\mathcal L$?

\medskip

The number of lattice points in $\Delta_d$ equals $\binom {d^2+3}{3}$.
In the the sparse case,  the number of lattice points of a given
lattice polygon $P$ can be computed via
a theorem of Ehrhart valid for any $n$ \cite{Ehr67}, 
which amounts to Pick's formula in the case $n=2$. 
Given a positive integer $t$, we denote by $tP$ the Minkowski sum of $P$ with itself $t$ times, i.e. 
$t P = \{ p_1 + \dots +
p_t, \, p_i \in P \, \text{ for } \, i=1, \dots, t\}$. 
The number of lattice points in $tP$ equals
\begin{equation}\label{eq:Ehr}
\# (tP \cap \ZZ^2) \, = \, \vol(P) t^2 + \frac 1 2 \vol_\ZZ(\partial P) t + 1,
\end{equation}
where $\vol_\ZZ(\partial P)$ denotes the number of lattice points in the boundary of $P$. 
In particular, $\#(P \cap \ZZ^2) = \vol(P) + \frac 1 2 \vol_\ZZ(\partial P)  + 1$.

The proof of the following result is straightforward:

\begin{lemma} \label{lem:size}
In case $f_0, \dots, f_3$ are generic polynomials of degree 
$d$ in $(s_1, s_2)$, the number of unknowns in the linear system
$\mathcal L$ in the coefficients of the implicit equation $F$ 
is $\binom {d^2+3}{3}$ (approximately $d^6/6$) and the number of 
equations is $\binom {d^3+2}{2}$ (approximately $d^6/2$).

For any lattice polygon $P$ and generic polynomials $f_i$ 
with Newton polytope $P$, the linear system $\mathcal L$ 
has $\binom {\vol_\ZZ(P) + 3}{3}$ (approximately ${\vol_\ZZ(P)^3}/ 6$) variables
and $\frac {\vol_\ZZ(P)^3} 2 +\frac {\vol_\ZZ(P)^2} 2 \vol_\ZZ(P) + 1$ 
equations (approximately ${\vol_\ZZ(P)^3}/ 2$).
\end{lemma}

We will see in Remark~\ref{rem:sizeMcl} of Section~\ref{sec:syzygies} 
that the size of the involved 
linear systems in the syzygy based methods is drastically smaller.

\section{The main algorithm based on linear syzygies}\label{sec:syzygies}

Our main result is Theorem~\ref{th:mainalgo}, which has a wide applicability.  
We distill and state it in naive terms,
 which do not call upon the more sophisticate tools recalled 
 in Section~\ref{sec:tgtools} and Appendix~\ref{sec:catools} required for its proof. This is why we postpone the
 detail of Hypotheses~\ref{hyp1} and~\ref{hyp1XP} until Section~\ref{sec:tgtools}. Our approach is 
 an inhomogeneous 
translation of the basic general algorithm for the sparse case in \cite{BDD08}, which were
inspired by the methods \cite{BuJo03} for classical homogeneous polynomials.

Before moving to the implicitization of rational surfaces, we recall the
practical approach of moving lines proposed by Sederberg and Chen \cite{SC95}
for the implicitization of planar curves.

\subsection{Curves}\label{ss:curves}

A planar rational curve $\Cc$ over a field $\kk$ is given as the image of a map 
\begin{eqnarray*}
 \kk^1 & \stackrel{f}{\dashrightarrow} & \kk^2 \\ s & \mapsto & 
 \left(\frac{f_1(s)}{f_0(s)},\frac{f_2(s)}{f_0(s)}\right),
\end{eqnarray*}
with $f_i \in \kk[s]$  polynomials of degree $d$ in $s$. 
We can assume without loss of generality that $\gcd(f_0,f_1,f_2)=1$. 
Remark that a {linear syzygy}  can be represented as a linear form 
$L = h_0T_0+h_1T_1+h_2T_2$ in the new variables $T =(T_0,T_1,T_2)$ with 
$h_i \in \kk[s]$ such that 
\[
 \sum_{i=0,1,2} h_i f_i =0.
\]
With this incarnation, a linear syzygy was termed a \emph{moving line}.
For any $\nu \in \NN$, consider the finite-dimensional $\kk$-vector space 
$\Syz(f)_\nu$ of linear syzygies satisfying 
$\deg(h_i) \le  \nu$, and call $N(\nu)$ its dimension.

\medskip
Pick a $\kk$-basis $h^i=(h^i_0, \dots, h^i_3), i = 1, \dots, N(\nu)$ of $\Syz(f)_\nu$.
Consider the monomial basis $\{1, s, \dots, s^\nu\}$ of polynomials 
in $s$ of degree at most $\nu$ and
write for  each syzygy $h^i$:
\begin{eqnarray*}
 L_i 	&=	& L_i (s, T) = \sum_{j=0,1,2} h_j^i(s) T_j = \sum_{j=0,1,2} 
 \paren{\sum_{k=0}^\nu c_{jk}^{i} s^k} T_j\\
	&=	& \sum_{k=0}^\nu \paren{\sum_{j=0,1,2} c_{jk}^i T_j} s^k.
\end{eqnarray*}
Let $M_\nu$ be the $ N(\nu) \times (\nu +1) $ matrix of coefficients 
of the $L_i$'s with respect to the basis
$\{1, s, \dots, s^\nu\}$:
\[
 M_\nu \, = \, \left(\sum_{j=0,1,2} c_{jk}^i T_j \right)_{i=1, \dots, 
 N(\nu), j = 0, \dots, \nu}.
\]
Observe that the variable $s$ has disappeared. This is the 
{\bf main elimination step}! 

\medskip

It is known that for $\nu \geq d-1$, the matrix $M_\nu$ is a square matrix 
with $\det(M_\nu)=F^{\deg(f)}$, where $F$ is an implicit equation of $\Cc$. In case 
$\nu \geq d$, then $M_\nu$ is a non-square matrix with more columns than rows, 
but still the greatest common divisor of its minors of maximal size equals $F^{\deg(f)}$. 
In both cases, for $\nu \geq d-1$, a point $P \in \PP^2$ lies on $\Cc$ iff the rank of $M_\nu(P)$ drops.

In other words, one can always represent the curve as a square matrix of linear syzygies, 
which gives a \emph{matrix representation} of the implicit equation. In principle, one could now 
actually calculate the implicit equation, but the matrix $M_\nu$ is easier to get and
well suited for numerical methods
\cite{ACGS07}. As we remarked in the surface case, testing whether 
a point $p$ lies on the curve only requires computing the rank of $M_\nu$ evaluated in $p$. 
Also, the singularities of $\mathcal C$ can
be read off from $M_\nu$ \cite{JG09,CKPU11, BDA10}.

In the absence of common zeros of $f_0, f_1, f_2$, it is
possible to find the implicit equation via a resultant computation.  
Note that for a parametrization with polynomials of degree $d$, the Sylvester 
resultant matrix uses a matrix of size $2d$, while the syzygy method uses 
$2$ matrices of size $d$, as the B\'ezout resultant.


\subsection{The general method of implicitization via linear syzygies for surfaces}
\label{ssec:general}

Assume we are given a rational parametrization of a surface $\Sc$  as in~\eqref{eq:f}.
We aim at finding a matrix representation for $\Sc$. Note that we can in principle assume
that $(f_0, \dots, f_3)$ are Laurent polynomials admitting negative exponents, but after
multiplying them by a common monomial, we get a new rational parametrization of $\Sc$
defined by polynomials $f_i \in \kk[s_1, s_2]$.
We will then assume, without loss of generality, 
that we are have a lattice polygon $P$ which lies in the first orthant of $\RR^2$
and contains the Newton polytopes of $f_0, \dots, f_3$.

We saw that in the curve case, it is always possible
to find a square  matrix representation. 
In the surface case, however, linear syzygies provide in general rectangular 
matrix representations and the implicit equation (raised to the degree of the map $f$) equals   
the great common divisor of the maximal minors (or the determinant of a complex).
A recent paper by Bus\'e \cite{buse13} presents a very interesting square matrix representation out of
a matrix representation $M$ when we work over the real numbers, by considering the square matrix $M M^t$. This approach is 
natural because of the properties of the rank of a real matrix with respect to its singular value decomposition. 
The determinant of $M M^t$ gives an implicit equation for $\Sc$ (in general, it gives $F$ with multiplicity),
which is moreover a sum of squares. As Bus\'e observes, for complex matrices it would be enough to
replace the transpose $M^t$ by the conjugate transpose.

The use of quadratic relations (i.e. linear syzygies among the products $f_if_j$ of any of two of 
the polynomials $f_i$ defining the parametrization) 
was proposed to construct square matrices \cite{SC95,CSC98,Co01,DA01,AHW05}. 
Khetan and D'Andrea generalized in 2006 \cite{KD06} the method 
of moving quadrics to the toric case. The choice of the quadratic syzygies is in general not canonical 
and the cost of computing syzygies is increased.
Note that syzygies in $(f_0, \dots, f_3)$ and the implicit equation $F$ have a common shape.
Indeed, linear syzygies $h=(h_0, \dots, h_3)$ of degree $\nu$ correspond to polynomials 
$H(s,T)=\sum_{i=0}^3 h_i(s) T_i$ such that $\sum_{i=0}^3 h_i(s) f_i(s) =0$, with $\deg(H)$ in
the $s$ variables equal to $\nu$, and $\deg(H)$ in the $T$ variables equal to $1$. 
Also, quadratic syzygies of degree $\nu'$ correspond to polynomials $H(s,T)=\sum_{i\le
j=0}^3 h_{i,j}(s) T_i T_j$ such that $\sum_{i,j=0}^3 h_{i,j}(s) f_{i} f_{j}(s)
=0$, with $\deg(H)$ in the $s$ variables equal to $\nu'$, and $\deg(H)$ in the
$T$ variables equal to $2$.  
The implicit equation (of degree $D$) is a polynomial
$H(s,T)=\sum_{|\alpha|\le D} h_\alpha T^\alpha$ such that $\sum_\alpha h_\alpha
f_1^{\alpha_1}(s) =0$. Thus, $\deg(H)$ in the $s$ variables equals $0$, and
$\deg(H)$ in $T$ variables equals $D$. So to go from linear syzygies to the implicit equation, in
some sense one has to play the
game of lowering the degree in the $s$ variables to
$0$ (which increases the degree in the $T$ variables up
to $D$).

We now present our main general algorithm to construct matrix representations
of parametrized surfaces. Clearly,  given any lattice polygon $P \subset \RR^2$,
$2 P =\{p_1 + p_2, p_i \in P\}$ is again a lattice polygon. Moreover,
in dimension two,
any lattice polygon is normal, which means that
$2 P \cap \ZZ^2 = \{p_1 + p_2, p_i \in P \cap \ZZ^2\}$, which is implicitly
used in the algorithm.

\begin{algorithm}\label{algo:mainalgo}
The following algorithm produces a matrix of polynomials in $(T_1, T_2, T_3)$
out of the input polynomials $f_0, \dots, f_3$ in variables $s=(s_1, s_2)$:

\begin{itemize}
\item {\bf INPUT:} A lattice polytope $P$ and 
polynomials $(f_0(s), f_1(s), f_2(s), f_3(s))$ with no common factor
and Newton polytopes $\Nc(f_i)$ contained in $P$.
\item {\bf STEP 1:} Consider syzygies $(h_0, \dots, h_3)$ with $\Nc(h_i) \subset
2 P$. Let
$(h^{(j)}_0,\dots,h^{(j)}_3)$, $j = 1, \dots, N$, be a $\kk$-basis of such
syzygies.
\item {\bf STEP 2:} Represent the syzygies as linear forms $L_j = h^{(j)}_0 T_0
+ \dots + h^{(j)}_3 T_3$. Write
$h^{(j)}_i = \sum_{\beta \in {2P} \cap \ZZ^2} h^{(j)}_{i, \beta} s^\beta$
and switch:
\[L_j = \sum_i h^{(j)}_i T_i  = \sum_\beta \left(
\sum_i  h^{(j)}_{i, \beta} T_i \right) s^\beta.\]
\item {\bf OUTPUT:} The matrix $M$ of linear forms
$\ell_{j,\beta}:=\sum_i  h^{(j)}_{i, \beta} T_i$.
\end{itemize}
\end{algorithm}

We illustrate the steps in Algorithm~\ref{algo:mainalgo} in the following example.

\begin{example}\label{example:IllustAlg}
 Let $P$ be the lattice polygon with vertices 
$(0,0)$, $(0,1)$, $(2,0)$ and $(1,1)$, with lattice points $p_0,\hdots,p_4$, as in
the figure.
We consider the following four polynomials with support in $P$,
where we denote $s:=(s_1,s_2)$, and given $p:=(i,j)$ we write $s^p:=s_1^is_2^j$:

\noindent\begin{minipage}{10cm}

\noindent $f_0 = 1 + 3 s_1 + s_1^2 + 2 s_2 + s_1 s_2 = s^{p_0}+ 3 s^{p_1}+ s^{p_2}+ 2s^{p_3}+s^{p_4}$,

\noindent $f_1 = 5 s^{p_0} - s^{p_1} - s^{p_2}+ 2s^{p_3} - s^{p_4}$,

\noindent $f_2 = 7 s^{p_0}+ 3 s^{p_1}+ 2 s^{p_2}+ 6 s^{p_3}+ 3 s^{p_4}$,

\noindent $f_3 =11 s^{p_0}+ 0 s^{p_1}+ 4 s^{p_2}+ 3 s^{p_3}+ 5 s^{p_4}$.
\end{minipage}
\begin{minipage}{2cm}
\begin{center}
   \includegraphics[scale=0.9]{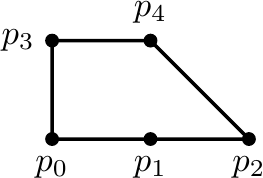}
  \label{fig:P}
\end{center}
\end{minipage}

To compute the syzygies in Step 1, we
consider the morphism $(a_0,a_1,a_2,a_3)\mapsto \sum_i a_if_i$, where $a_i$ are polynomials with support in $2P$.
Let $B$ be the matrix of this map in the monomials bases.
Since $2P$ has $12$ lattice points and $ \sum_i a_if_i$ has support in $3P$, which has $22$ lattice points, 
then $B$ is a \textsl{Sylvester} matrix of size $22\times 48$. It can be easily checked that $B$
is full ranked (same as for generic polynomials). Thus, the kernel of $B$ has dimension $N=48-22=26$, 
which is the number of linearly independent syzygies.

To construct the matrix $M$, assume that we choose as our first syzygy 
the following $4$-tuple of polynomials $(h^{(1)}_0, \dots, h^{(1)}_3)$ with $\Nc(h^{(1)}_i) \subset
2 P$:

\

$h^{(1)}_0= -196s^{2p_0}+504s^{p_0+p_1}-257s^{p_0+p_2}+672s^{p_0+p_3}+234s^{p_0+p_4}$,

$h^{(1)}_1=-237s^{p_0+p_2}+420s^{p_0+p_3}-168s^{p_0+p_4}$,

$h^{(1)}_2=28s^{2p_0}+10s^{p_0+p_2}-364s^{p_0+p_3}+226s^{p_0+p_4}$, and

$h^{(1)}_3=-216s^{p_0+p_4}$.

\

We consider $L_1 = h^{(1)}_0 T_0
+ \dots + h^{(1)}_3 T_3$ and we write
\[L_1 = (-196T_0+28T_2)s^{2p_0}
+(504T_0)s^{p_0+p_1}
+(-257T_0-237T_1+10T_2)s^{p_0+p_2}\]
\[
+(672T_0+420T_1-364T_2)s^{p_0+p_3}
+(234T_0-168T_1+226T_2-216T_3)s^{p_0+p_4},\]
which gives the first column of the 
$28\times 12$-matrix $M$ (computed with Macaulay2 computer-algebra software \cite{M2})

{\footnotesize\begin{verbatim}
| -196T_0+28T_2               0                           0                          ...
| 504T_0                      -196T_0+28T_2               0                          ...
| -257T_0-237T_1+10T_2        504T_0                      -196T_0+28T_2              ...
| 672T_0+420T_1-364T_2        0                           0                          ...
| 234T_0-168T_1+226T_2-216T_3 672T_0+420T_1-364T_2        0                          ...
| 0                           -257T_0-237T_1+10T_2        504T_0                     ...
| 0                           234T_0-168T_1+226T_2-216T_3 672T_0+420T_1-364T_2       ...
| 0                           0                           -257T_0-237T_1+10T_2       ...
| 0                           0                           234T_0-168T_1+226T_2-216T_3...
| 0                           0                           0                          ...
| 0                           0                           0                          ...
| 0                           0                           0                          ...

                           12                          26
Matrix (QQ[T , T , T , T ])   <--- (QQ[T , T , T , T ])
            0   1   2   3               0   1   2   3
\end{verbatim}}

The columns of this matrix $M$ are given by a choice of a basis of syzygies with support in $2 P$.
The corresponding linear forms $L_1, \dots, L_N$ are known as the \emph{moving 
planes} defining the surface
parametrized by $f_1, \dots, f_3$.
The associated rational map $f$ has  $\deg(f)=1$. It can be checked that 
the common factor of any maximal minor of $M$ is the degree $3$ implicit equation of the closed
image of $f$:
\[
 F= 2643T_0^3+2905T_0^2T_1+1345T_0T_1^2+91T_1^3-8T_0^2T_2-444T_0T_1T_2+284T_1^2T_2+\cdots,
\]
as asserted by Theorem~\ref{th:mainalgo} below.

Note that we have to write the lattice points in $2P$ as a sum of two points in $P$, but in general
there is not a unique way of doing this. In our example, for instance, $p_0+p_2= p_1+p_1$, so 
a choice was made. In fact, it is possible to make a coherent choice in general with the use of weight vectors,
but any choice will work since in the quotient ring $A$ defined in~\eqref{eq:A} below, it holds that
$X_0.X_2$ and $X_1^2$ are identified.
\end{example}

We now state our main result. The proof will be given in Section~\ref{sec:tgtools}.

\begin{theorem}\label{th:mainalgo}
Given  $(f_0(s), f_1(s), f_2(s), f_3(s))$ with no common factor, with Newton polytopes
contained in $P$  and satisfying hypotheses~\ref{hyp1} below,\theoremstyle{remark}
Algorithm~\ref{algo:mainalgo} computes a presentation matrix of the implicit equation
of the rational map $f$. That is, the rank of the matrix $M$ drops precisely when evaluated
at the points in the closure of the image of $f$. 

Moreover, the implicit equation $F$
can be computed as
\begin{equation}\label{eq:deg}
F^{\deg(f)} \, = \, {\rm gcd}(\mbox{maximal  minors  of }  M).
\end{equation}
\end{theorem}

The main ingredient for the validity of Algorithm~\ref{algo:mainalgo} to give
a matrix representation
is the choice  ${\mathbf 2P} \cap \ZZ^2$
of the support of the linear syzygies.
Again, the ``instant'' elimination is done in STEP 2, where the $s$ variables
give the monomial basis which is used to compute the matrix $M$ (and thus they
disappear from the output!).

In fact, Algorithm~\ref{th:mainalgo} can be run 
without checking Hypotheses~\ref{hyp1}.
We point out in Remark~\ref{rem:whatifnot} the possible outcomes. 
The general algorithm can be refined using Theorem~11 in \cite{BDD08}.

\begin{theorem}\label{th:refinement}
Assume $f$ satisfies the hypotheses of Theorem~\ref{th:mainalgo}. 
If the lattice polygon $P$ can be written as a multiple $P = dP'$ of another another lattice polygon
$P'$  \emph{without
interior lattice points}, then we can consider in {\bf STEP 1}  of Algorithm~\ref{algo:mainalgo} 
syzygies $(h_0,\dots, h_3)$ with smaller support
$N(h_i)$ contained in $(2d-1)P'$ (which is strictly contained in $2 P$), and 
the {\bf OUTPUT} will still be a  matrix 
representation for $f$. 
Moreover, in case $P'$ is the unit simplex, it is enough to consider syzygies with support
inside $(2d-2)P'$.
\end{theorem}

We then have the following comparison between the general syzygy method and the
naive linear algebra method described in Section~\ref{sec:naive}.

\begin{remark}\label{rem:sizeMcl}
Assume that $P$ is the triangle of size $d$. 
Then, as it is enough to consider syzygies 
of degree $2d-2$, they can be found by solving a linear system 
on $ 4 \binom{2d}{2}$ variables with $\binom {3d}{2}$ equations. 
That is, both sizes, as well as the vector space dimension of 
the space of syzygies in this degree, are quadratic in $d$.
The matrix $M$ has then a number of rows quadratic in $d$. 
The number of its columns equals $\binom{2d}{2}$, again quadratic in $d$. Comparing
with the sizes in Lemma~\ref{lem:size}, which are of degree $6$ in $d$,
we observe that the syzygy method is a great improvement on the naive linear algebra method!
\end{remark}

The same improvement occurs for any lattice polygon $P$. Using~\eqref{eq:Ehr}, we see that
syzygies with support in $2P$ can be obtained by solving a system with approximately
$9 \, {\rm vol}(P)$ equations in $16 \, {\rm vol}(P)$ variables and both row and column sizes
of the  matrix representation $M$ are of the order of ${\rm vol}(P)$ and not of its
cube, as in Lemma~\ref{lem:size}.

\subsection{The bihomogeneous case and beyond}\label{ssec:bihom}

As we have mentioned, the main motivation for the implicitization problem comes from Computer Aided
Geometric Design and geometric modeling. In this area, bihomogeneous surfaces (corresponding to
rectangular support $P$) are known as tensor product surfaces, and they  play a central role, 
in particular the B\'ezier surfaces.  
Quoting Ulrich Dietz~\cite{die98}: ``\textit{In current CAD systems tensor product surface 
representations with their rectangular structure are a de facto standard}".
These surfaces (called \emph{NURBS}) are given by pieces of
parametrized surfaces cut by curves. So, it is necessary to control the location of the parameter,
which can be achieved by 
computing the kernel of the matrix representation we give,
as explained in~\cite{buse13}.

Due to the nature of the base locus of the parametrization, 
many of the current geometric modeling systems do not satisfy the hypotheses to be detailed
in~\ref{hyp1}
needed for Theorem~\ref{th:mainalgo} to hold, if considered as homogeneous polynomials
(with $P$ an equilateral triangle). But if we use a rectangle $P$ as the input in
Algorithm~\ref{algo:mainalgo},  it is
possible to get a full-ranked  matrix representation by Theorem~\ref{th:mainalgo}.
In this bihomogeneous case, 
the detailed study of regularity in \cite{BC10} allows to get the following improvement
in the support of the proposed linear syzygies in STEP 1 of Algorithm~\ref{algo:mainalgo}: 
it is enough that the support of these syzygies is contained in a polygon obtained by only
enlarging the rectangle support $P$ of the input polynomials (approximately) 
to its double in \emph{only} the horizontal 
\emph{or} the vertical direction, instead of considering syzygies with support in 
(the lattice points of) $2P$.

A more general result 
can be obtained for bigraded toric surfaces, and in particular for
lattice polygons defining a
Hirzebruch surface, that it, for
Hirzebruch quadrilaterals $H_{a,b,n}$ with vertices $(0,0), (a,0), (0,b)$
and $(a+nb,b)$, for any $a,b n \in \NN$. We state this extension in Theorem~\ref{th:Hirzebruch} below.

\medskip

\begin{algorithm}\label{algo:Hirzebruch}
Take as {\bf INPUT} a Hirzebruch lattice polygon $P= H_{a,b,n}$
and bivariate polynomials  $(f_0, f_1, f_2, f_3)$ with Newton polygons
contained in $H_{a,b,n}$, and which 
satisfy the hypotheses~\ref{hyp1XP}. Run algorithm~\ref{algo:mainalgo} with the
following modification: in {\bf STEP 1} consider a basis of syzygies with support in the
smaller lattice quadrilaterals $H_{2a-1,b-1,n}$ (or  $H_{a-1,2b-1,n}$ instead). 
The {\bf OUTPUT} is the corresponding 
matrix $M$ of linear forms.
\end{algorithm}

In most cases, it is convenient to consider syzygies with support in $H_{2a-1,b-1,n}$ 
rather than in $H_{a-1,2b-1,n}$ since the first one has less lattice points. 

\noindent\begin{minipage}{7cm}
In general, a Hirzebruch lattice polygon $H_{x,y,n}$ has the shape in the diagram on the right.
\end{minipage}
\begin{minipage}{5cm}
\begin{center}
   \includegraphics[scale=0.9]{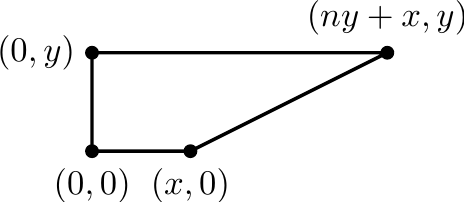}
  \label{fig:H}
\end{center}
\end{minipage}

\begin{remark}\label{rem:HirzBigrad}
Note that for $n=0$, $H_{a,b,0}$ is the standard lattice rectangle with vertices in $(0,0),(a,0),(0,b), (a,b)$, 
and thus Algorithm \ref{algo:Hirzebruch} 
works in particular in a standard bihomogeneous setting.
\end{remark}

\begin{theorem}\label{th:Hirzebruch}
Given  $(f_0(s), f_1(s), f_2(s), f_3(s))$ with no common factor, with Newton polytopes
contained in a $H_{a,b,n}$  and satisfying hypotheses~\ref{hyp1XP} below,
Algorithm~\ref{algo:Hirzebruch} computes a presentation matrix of the implicit equation
of the rational map $f$. That is, the rank of the output matrix $M$ drops precisely when evaluated
at the points in the closure of the image of $f$. 

Moreover, the implicit equation $F$
can be computed as
\begin{equation}\label{eq:deg2}
F^{\deg(f)} \, = \, {\rm gcd}(\mbox{maximal  minors  of }  M).
\end{equation}
\end{theorem}

The proof of Theorem~\ref{th:Hirzebruch} will be also given in section~\ref{sec:tgtools}.


\section{The hypotheses via toric geometry and the proofs of our main results}\label{sec:tgtools}

In this section we will recall a minimum of theoretical tools 
from toric geometry in order to
state the hypotheses needed for Theorems~\ref{th:mainalgo} and~\ref{th:Hirzebruch}
and to give their proofs. The main homological commutative  algebra tools that are the core
of the proofs are recalled in Appendix~\ref{sec:catools}.

We refer to \cite{Co03b, Fu93,CLS11} and  \cite[Ch.5\&6]{GKZ94} for the general notions, and to \cite[\S 2]{KD06},
\cite{BDD08,Bot09} for applications to the implicitization problem.
Any reader only interested in the application of Algorithm~\ref{algo:mainalgo} or
its bihomogeneous (toric) refinement given in Algorithm~\ref{algo:Hirzebruch} can skip this section.

As usual, we denote by $\kk^*=\kk \setminus \{0\}$ the multiplicative group of units of $\kk$.
The first observation is that we can equivalently consider our parametrization~\eqref{eq:f} 
as a map $\tilde{f}: \kk^2 \dashrightarrow 
 \PP^3(\kk)$ or $\tilde{f}: (\kk^*)^2 \dashrightarrow 
 \PP^3(\kk)$ with image inside $3$-dimensional projective space, and domain a dense open set $U$ in affine space
 $\kk^2$ or the torus $(\kk^*)^2$ over $\kk$, given by  
\begin{equation}\label{eq:of}
 s \mapsto (f_0(s): f_1(s) : f_2(s) : f_3 (s)),
\end{equation}
for any $s \in U$, and we have the commutative diagram
\begin{equation}\label{diagram1}
\xymatrix{
(\kk^*)^2  \ar@{-->}[r]^{f} \ar@{-->}[rd]_{\tilde{f}} & \kk^3 \ar@{^(->}[d]_\iota\\
 &\PP^3.}
\end{equation}
In fact, if $F$ is the implicit equation of the (closure of the) 
image of $f$, the (closure of the) image of $\tilde{f}$ is the closure
of  $\Sc$ under the standard embedding $\kk^3 \hookrightarrow \PP^3(\kk)$ 
and its equation is the homogenization of $F$. 

Similarly,  we can 
consider our rational parametrization from any algebraic variety 
which contains the domain of $f$ as a \emph{dense} subset. We will choose embedded or
abstract compact toric varieties to get a degree or multidegree notion that will
allow us to get homological arguments to ``bound'' the support of the syzygies
in Theorem~\ref{th:mainalgo} and in~\ref{th:Hirzebruch}.

\subsection{Toric embeddings}\label{ssec:toricemb}
Let $\tilde{f}$ be a rational map
as in~\eqref{eq:of}.
The base points of the parametrization are the common zeros of $f_0, \dots, f_3$,
that is, the points where the map is not defined.
We assume that $\tilde{f}$ is a generically finite map onto its image and 
hence it para\-me\-trizes an irreducible surface $\Sc \subset \PP^3$. We also assume
without loss of generality 
that $\gcd(f_0,f_1,f_2,f_3)=1$, which means that there are only finitely many base points.

Let $P \subset \RR^2$ be a lattice polygon with $m+1$ lattice points,
which contains the Newton polygons $\Nc(f_0),\dots,\Nc(f_3)$.
Write $P \cap \ZZ^2 = \{p_0, \dots, p_m\}$.
The polygon $P$ determines a projective toric surface
$\Tc_P \subseteq \PP^m$ as the closed image of the embedding
\begin{eqnarray*}
 (\kk^*)^2   &  \stackrel{\rho}{\rightarrow}  & \PP^m \\
(s_1,s_2) & \mapsto & (\ldots : s^{p_i}  : \ldots)
\end{eqnarray*}
where $i=0, \dots,m$. 
For example, the unit triangle with vertices $(0,1)$, $(1,0)$ 
and $(0,0)$ (or any lattice translate of it) corresponds to $\PP^2$, and any
lattice rectangle gives a Segre-Veronese projective embedding of 
$\PP^1 \times \PP^1$, which are special cases of toric embeddings.

\begin{example}\label{ex:p1p1}
Assume $P$ is the unit square, with $m+1=4$ integer points: 
\[
 p_0 =(0,0), p_1 = (1,0), p_2 = (0,1), p_3 = (1,1).
\]
A polynomial $f_i$ with Newton polytope contained in $P$ looks like
\begin{equation}\label{eq:p1p1}
f_i(s) = a_{(0,0)} + a_{(1,0)} s_1 + a_{(0,1)} s_2 + a_{(1,1)} s_1 s_2.
\end{equation}
We take $4$ new variables $(X_0: X_1 : X_2 : X_3)$  as the homogeneous coordinates
in $\PP^3$. The toric variety $\Tc_P$ is the projective variety in $\PP^3$ cut out
by the relation  
$X_0 X_3 - X_1 X_2=0$. This binomial equation comes from the primitive affine relation $p_0+p_3= p_1+p_2$,
which implies the multiplicative relation
$ s^{p_0} s^{p_3} = s^{p_1}s^{p_2}$ between the monomials with these exponents.
The coordinate ring of $\Tc_P$ is the quotient ring
$\kk[X_0,\dots,X_3]/\langle X_0 X_3 - X_1 X_2\rangle.$
\end{example}

In general, we will call $(X_0: \dots : X_m)$ the homogeneous coordinate variables
in $\PP^m$. Write $P \cap \ZZ^2 = \{p_0, \dots, p_m\}$.
We set one variable $X_i$ for each integer point $p_i$ in $P$ and we record 
multiplicatively (by binomial equations) the affine relations among these points. 
These binomials generate the toric ideal $J_P =J(\Tc_P)$, which defines the
variety $\Tc_P \subset \PP^m$.
To each 
\begin{equation}\label{eq:fisP}
f_i(s) = \sum_{i=0}^m a_{p_i} s^{p_i},
\end{equation}
we associate the homogeneous linear form 
\begin{equation}\label{eq:gis}
 g_i(s) = \sum_{i=0}^m a_{p_i} X_i.
\end{equation}
For instance, in Example~\ref{ex:p1p1}, the polynomial $f_i$ gets translated to
\[
g_i(X_0,\dots, X_3) = a_{(0,0)}X_0 + a_{(1,0)}X_1 + a_{(0,1)}X_2 + a_{(1,1)}X_3, \]
and over $\Tc_P$, we have the relation $X_0 X_3 - X_1X_2=0$.

The rational map $\tilde{f}$ factorizes through $\Tc_P$ in the following way
\begin{equation}\label{diagram2T}
\xymatrix{
 (\kk^*)^2 \ar@{-->}[r]^{\tilde{f}} \ar@{^(->}[d]^{\rho} & \PP^3 \\
\Tc_P \ar@{-->}[ur]_g }
\end{equation}
where $g=(g_0:g_1:g_2:g_3)$ is given by four homogeneous linear polynomials $g_0,g_1,g_2,g_3$ 
in $m+1$ variables. Thus, we have a new homogeneous 
parametrization $g$ of the closed image of $\tilde{f}$ from $\Tc_P$. The polynomials $g_i$ generate
an ideal $I$ in the coordinate ring
\begin{equation}
 \label{eq:A}
A=\kk[X_0,\ldots,X_m] / J_P 
\end{equation}
of $\Tc_P$. This ideal $I$ defines the structure of the base locus in $\Tc_P$.

The embedding $\rho: (\kk^*)^2 \rightarrow \PP^3$ provides a $\ZZ$-grading in the coordinate 
ring $A$ of $\Tc_P$, which is used to study the map $g$ with the tools recalled in
Appendix~\ref{sec:catools}.
 
\subsection{Abstract toric varieties and Cox rings}
Given a lattice polygon $P$, one can also associate to it an abstract compact toric variety $X_P$ that
naturally contains the torus $(\kk^*)^2$ as a dense open set (via the map we call $j$ below), 
adjoining a torus invariant divisor to each edge of $P$. We refer the reader to \cite{CLS11,CLO98}
for the theory and details.

The map $\tilde{f}$ also defines a rational map
$\overline{f}$ that makes the following diagram commutative:
\begin{equation}\label{eq:diagram2}
\xymatrix{
X_P  \ar@{-->}[rd]^{\overline{f}}& \\
(\kk^*)^2  \ar@{^(->}[u]_{j}\ar@{^(->}[d]_{\rho}\ar@{-->}[r]_{\tilde{f}} & \PP^3 \\
 \Tc_P \ar@{-->}[ur]_{g}& }
\end{equation}

\begin{example}\label{ex:p1p1b}
Assume $P$ is the unit square, with $N=4$ edges: the segments $E_1=[(0,0),(1,0)]$,
$E_2=[(0,0),(0,1)]$, $E_3=[(0,1),(1,1)]$, and $E_4 = [(1,0),(1,1)]$. The respective inner normal
vectors $\eta_1=(0,1), \eta_2=(1,0)$, $\eta_3=(0,-1), \eta_4=(-1,0)$ satisfy the linear
relations $\eta_1+\eta_3=0$, $\eta_2+\eta_4=0$, which give rise to two homogeneities.
We introduce four associated variables $Y=(Y_1, \dots, Y_4)$.
A polynomial $f_i$ with Newton polytope $P$ as in~\eqref{eq:p1p1} defines
a bihomogeneous polynomial (in $(Y_1, Y_3)$ and $(Y_2,Y_4)$):
\[
\overline{f}_i(Y) = a_{(0,0)} Y_3 Y_4 + a_{(1,0)} Y_1 Y_4+ a_{(0,1)} Y_2 Y_3 + a_{(1,1)} Y_1 Y_2.
\]
These polynomials $\overline{f}_i$ define the map $\overline{f} = (\overline{f}_0:\dots,\overline{f}_3)$.
\end{example}

The main motivation for this change of perspective comes again from the commutative algebra
results needed for the proof of Theorems~\ref{th:mainalgo} and~\ref{algo:Hirzebruch}. The Cox ring of $X_P$
is endowed with a more natural multigrading, which is finer than the grading obtained via the
embedded projective variety $\Tc_P$. Also, this point of view has an impact in the
computations, as the number of variables to eliminate is smaller (one for each edge of $P$, instead of one
for each lattice point in $P$). In our small example~\ref{ex:p1p1b}, there are four edges and
four lattice points, but the number of edges can remain constant while the number of lattice
points goes to infinity.

\subsection{Precise hypotheses and proof of Theorem~\ref{th:mainalgo}}

In this subsection we detail the precise hypotheses that ensure the validity of 
Theorem~\ref{th:mainalgo} and we prove it, based on results in~\cite{BDD08}.
We first need to recall a few standard definitions from commutative algebra.

\begin{definition}
 Given (nonzero) homogeneous polynomials $(g_0, \dots, g_3)$, 
defining a rational map
$g: \Tc_P \dashrightarrow \PP^3$ as in~\eqref{diagram2T}, a point $p \in \Tc_P$ is a base point of $g$ if it 
 is a common zero set of $g_0, \dots, g_3$, that is, if $p$ is a zero of the ideal $I \subset A$ in $\Tc_P$.

 Let $p \in \Tc_P$ be a base point of $g$. 
 The local ring of $p$ is the ring 
 $A_p = \{ h_1/h_2 , \, h_i \in A, \, h_2(p) \not=0\}$,
with the natural operations induced from $A$  (in turn, naturally induced from the polynomial ring).
Let $I_p$ be the ideal generated by
(the classes of) $g_0, \dots g_3$  in $A_p$.
We say that $p$ is a local complete intersection base point
if $I_p$ can be generated by only $2$ elements. 
We say that $p$ is an almost complete intersection base point if $I_p$ can be
generated with $3$ elements.

We have similar definitions for the map $\tilde{f}: (\kk^*)^2 \dashrightarrow \PP^3$.
\end{definition}

For a given lattice polygon $P$, here are the hypotheses we need in terms of $g$:

\begin{hyp}\label{hyp1}
There are only finitely many 
 base points of $g$ on $\Tc_P$ which are local complete
intersections.
\end{hyp}

We cannot easily find hypotheses on $f$ equivalent to Hypotheses~\ref{hyp1}. 
Given a lattice polygon $P$, an edge $E$ of $P$, and a polynomial $f_i$ 
with $\Nc(f_i)$ contained in $P$ as in~\eqref{eq:fisP}, the restriction  ${f_i}_{|E}$ of $f_i$ 
to $E$ is defined as the sub-sum of the monomials with exponents $p_i$ in $E$. We have the
following partial translation.

\begin{proposition}\label{prop:inequiv}
Let $f,\Tc$ and $g$ be as in~\eqref{diagram2T}.
Then
\begin{enumerate}
\item \label{one} There are only finitely many 
base points of $g$ on $\Tc_P$ if and only if 
there are only finitely many 
isolated base points of $f$ in the torus and for each edge $E$ of $P$, at least one of the restrictions
${f_i}_{|E}$ is nonzero.
\item \label{two} If $g$ has finitely many 
isolated base points on $\Tc_P$ which are local complete
intersections, then the base points of $f$ in the torus are local complete
intersections.
\end{enumerate}
 \end{proposition}

\begin{proof}
 The map $\rho$ defines an isomorphism between $(\kk^*)^2$ and its image (which is an open dense subset
 of $\Tc_P$), sending a base point
 $q$ of $\tilde{f}$ (that is, a point where $f_0(q)=\dots =f_3(q)=0$)
to a base point $p=\rho(q)$ of $g$, and reciprocally, any base point $p$ of $g$ in $\rho(\kk)^2$
is the image of a base point $q$ of $f$.
Moreover, we have an isomorphism between the ideal generated by $f_0, \dots, f_3$ at $q$ and $I_p$.
Any base point of $g$ outside the image of $\rho$ cannot be seen in the torus. But these points are either the
fixed torus points corresponding to the finitely many vertices of $P$, or they lie in the torus
of the toric divisor $D_E$ in $\Tc_P$ associated to an edge $E$ of $P$. As $D_E$ has dimension $1$,
there are finitely many solutions as long as at least one of the ${f_i}_{|E}$ is nonzero.
  \end{proof}
 
 The following example shows that the converse of item~\eqref{two} in Proposition~\ref{prop:inequiv}
does not hold.

\begin{example}\label{ex:nonlaci}
Consider the parametrization with $6$ monomials: $(f_0,f_1,f_2,f_3)=(st^6+2,st^5-3st^3,st^4+5s^2t^6,2+s^2t^6)$.
Then, $f$ has no base points in the torus. But if we consider their standard homogenizations to degree $8$
polynomials (that is, we take $P$ equal to $8$ times the standard unit simplex in the plane), the corresponding
homogeneous polynomials $g_0, \dots,g_3$ have one base point ``at infinity'' 
which is not even an almost locally complete intersection.
\end{example}

We now give the proof of Theorem~\ref{th:mainalgo}.

\begin{proof}[Proof of Theorem~\ref{th:mainalgo}]
Given  $(f_0(s), f_1(s), f_2(s), f_3(s))$ with no common factor and Newton polytopes contained in $P$, the 
corresponding polynomials $g_i$ associated to $f_i$ are homogeneous of degree $d=1$ and satisfy Hypotheses~\ref{hyp1}.

From \cite[Cor.\ 14]{BDD08} one has that for $d=1$,  the matrix of coefficients of a $K$-basis of 
the module of Syzygies of $g$ in any degree $\nu\geq 2$ with respect to a $K$-basis of the graded piece  $A_\nu$
of $A$, is a  matrix representation for the closure of the image of $f$, which equals the closure of the image
of $g$. 

In particular, we can take $\nu=2$. In STEP 1 of Algorithm~\ref{algo:mainalgo}, the syzygies 
$(h^{(j)}_0,\dots,h^{(j)}_3)$ for $j = 1, \dots, N$  with $\Nc(h^{(j)}_i) \subset 2 P$ for all $i,j$, provide a $\kk$-basis of 
the module of syzygies of $g$ in degree $2$, since classes of monomials of degree $2$ in $A$ correspond to 
monomials in the $s$ variables with exponents in $2P$.

Equality~\eqref{eq:deg} follows from Theorem 13 in~\cite{BDD08}.
\end{proof}

In principle, given a rational map $\tilde{f}$, we could take any lattice polygon $P$
containing the union $\Nc(f)$  of Newton polytopes of $f_0, \dots, f_3$. 
Note that the hypothesis that $f$ is generically finite implies 
that $\Nc(f)$ is two-dimensional. Taking $P$ strictly containing $\Nc(f)$
will increase the number of exponents and will in general produce bad behaviour
of $g$ at the fixed points in $\Tc_P$ corresponding to the vertices of $P$ which do not
lie in $\Nc(f)$. 

\begin{remark}\label{rem:whatifnot}
We can check algorithmically if 
$f_0, \dots, f_3$ have
finitely many solutions over $(\overline{\kk}^*)^2$ and if for any edge $E$ at least one of the restrictions
${f_i}_{|E}$ is nonzero. So, by Proposition~\ref{prop:inequiv},   
we can check whether $g$ has finitely many base points in 
$\Tc_P$. 

Assume the dimension of the base locus of $g$ is zero.
As we remarked in Example~\ref{ex:nonlaci}, even if we could check the 
local behavior of the base points of $f$
in the torus, this would not imply the satisfiability of Hypotheses~\ref{hyp1}.
But what if we don't check this and run Algorithm~\ref{algo:mainalgo}?  \dots
nothing bad!

We then check whether the output matrix $M$ in has full rank.

\begin{itemize}
\item If the rank of $M$ is not maximal, then there is at least one base point $p$ of $g$
which is not an almost local complete intersection. 
In this case, we cannot get the implicit equation, but we get a 
certificate of the bad behavior of the base locus (without
computing it).

\item If the rank of $M$ is maximal, it may happen that the its rank
drops when evaluated at points outside $\Sc$ due
to the existence of an almost complete intersection but non complete intersection base point. In this case, 
the greatest common divisor of the maximal minors of $M$ would have irreducible 
factors other than the implicit equation $F$. In fact, the existence of other irreducible 
factors is equivalent to the fact that there exists a base point which is
an almost local complete intersection but not a local complete intersection.

\end{itemize}
\end{remark}

\subsection{The hypotheses and proof of Theorem~\ref{th:Hirzebruch}}
In this subsection we detail the precise hypotheses that ensure the validity of 
Theorem~\ref{th:Hirzebruch} and we prove it, based on results in~\cite{Bot10}.

Given $P$ and $f$, here are the hypotheses we need in terms of the map $\overline{f}$
in~\eqref{eq:diagram2}:

\begin{hyp}\label{hyp1XP}
There are only finitely many 
base points of $\overline{f}$ on $X_P$ which are local complete
intersections.
\end{hyp}

Again, we cannot easily find hypotheses on $f$ equivalent to Hypotheses~\ref{hyp1},
since good algebraic behaviour of the base points in the torus does not imply the
same behaviour for the possible base points of $\overline{f}$ at the invariant
divisors in $X_P$ associated with the edges of $P$.

\begin{proposition}\label{prop:inequivXP}
Let $f,X_P$ and $\overline{f}$ as in~\ref{eq:diagram2}.
Then
\begin{enumerate}
 \item \label{oneb} There are only finitely many 
isolated base points of $\overline{f}$ on $X_P$ if and only if 
there are only finitely many 
isolated base points of $f$ in the torus and for each edge of $P$, at least one of the restrictions
of the $f_i$ is nonzero.
\item \label{twob} If $\overline{f}$ has finitely many 
base points on $X_P$ which are local complete
intersections, then the base points of $f$ in the torus are local complete
intersections.
\end{enumerate}
 \end{proposition}

We next give the proof of Theorem~\ref{th:Hirzebruch}.

\begin{proof}[Proof of Theorem~\ref{th:Hirzebruch}]
By hypothesis, there are only finitely many isolated base 
points of 
$\overline{f}$ on the toric variety $X_P$ associated with $P:=H_{a,b,n}$,  which are local complete intersections. 
There are four primitive inner normal vectors of $P$: $\eta_1= (0,1), \eta_2 = (0,1), \eta_3 = (-1,0), \eta_4 = (-1,n)$,
which satisfy the linear relations $\eta_3 = - \eta_1, \eta_4 = n \eta_2 - \eta_1$. So any multidegree $\nu$ can be described
by a ``bidegree'' $(\nu_1, \nu_2)$ given by the degrees with respect to the first normals and which fixes (up to translation) the 
associated polytope $P_\nu$ with the same normals as $P$.
Thus, by \cite[Thm.\ 5.5]{Bot10} the matrix of 
coefficients of a $K$-basis of 
the module of Syzygies of $\overline{f}$ in any bidegree $(\nu_1,\nu_2)$ with $\nu_1\geq 2a-1$ and $\nu_2\geq b-1$ 
(or $\nu_1 \geq a -1, \nu_2 \ge 2 b-1$).
\footnote{The choice of the 
bidegree is less obvious than in the graded case. For further details, see definition of $\mathfrak R_B(\gamma)$ in \cite[Thm.\ 
5.5]{Bot10}, or the analysis of the bidegree in the standard bigraded case in \cite[Sec.\ 7.1]{Bot10}} 
with respect to a $K$-basis 
of the bigraded piece ${(\nu_1,\nu_2)}$ of the Cox ring of $X_P$, 
is a  matrix representation for the closure of the image 
$\Sc$ of $f$ (which equals the closure of the image of $\overline{f}$). 

Taking $(\nu_1,\nu_2)=(2a-1,b-1)$ one has that in STEP 1 a basis of syzygies $(h^{(j)}_0,\dots,h^{(j)}_3)$, for $j = 1, \dots, 
N$ with all $\Nc(h^{(j)}_i)$ with support in the
quadrilateral 
$H_{2a-1,b-1,n}$, provides a $\kk$-basis of the module of 
syzygies of $\overline{f}$ in bidegree $(2a-1,b-1)$. Hence, the matrix $M$ 
of coefficients of such syzygies obtained 
in STEP 2 gives a representation matrix for $\Sc$.

Equality~\eqref{eq:deg2} also follows from Theorem~5.5 in~\cite{Bot10}. 
\end{proof}


\section{Examples}\label{sec:examples}

This section consists of four examples which highlight the usefulness of our approach.
Example \ref{Ex:ToricNotP2} is taken from 
a case studied in \cite[Ex.\ $18$]{BDD08} 
of a sparse parametrization  where projective implicitization does not work due to the nature of the base locus of the 
map, but Algorithm \ref{algo:mainalgo} is applicable with a right choice of polygon $P$ read from the monomials of
the input polynomials. In Example 
\ref{Ex:HighDeg} we show how the method in Algorithm \ref{algo:mainalgo} works for a parametrization given by fewnomials of high 
degree, where classical resultant tools fail due to the computational complexity. In Example~\ref{Ex:BPtorus},
classical resultant tools  fail because of the existence of a base point in the torus.
Finally, in Example \ref{Ex:WithoutEmb} we compare the methods in Algorithms \ref{algo:mainalgo} and \ref{algo:Hirzebruch}.

\subsection{A very sparse parametrization}\label{Ex:ToricNotP2}

Consider the parametrization with $6$ monomials given in \cite[Ex.\ $18$]{BDD08}:
$(f_0,f_1,f_2,f_3)=(st^6+2,st^5-3st^3,st^4+5s^2t^6,2+s^2t^6)$.
The matrix representation can be computed using the package \textit{MatrixRepToric.m2} 
\cite{BD10}
in the computer algebra software \textit{Macaulay2} \cite{M2}.

One first defines the map $f$ given by polynomials in the ring $S= \QQ[s,t]$
(note that for easiness of typing, we call the variables $(s,t)$ instead of $(s_1, s_2)$):

\begin{verbatim}
S = QQ[s,t];
f = {s*t^6+2, s*t^5-3*s*t^3, s*t^4+5*s^2*t^6, 2+s^2*t^6};
\end{verbatim}

\noindent\begin{minipage}{8cm}
Consider $P$ the lattice triangle with vertices $(0,0), (1,6)$ and $(2,6)$.

\medskip

One can compute $P$ by the command:
 \begin{verbatim}
P = polynomialsToPolytope L
 \end{verbatim}

 The lattice-points of $P$ can be computed using the auxiliary Macaulay2 package \texttt{Polyhedra} as:
 \begin{verbatim}
latticePoints P
\end{verbatim}
 
 \end{minipage}
\begin{minipage}{3cm}
 \begin{center}
  \includegraphics[scale=.6]{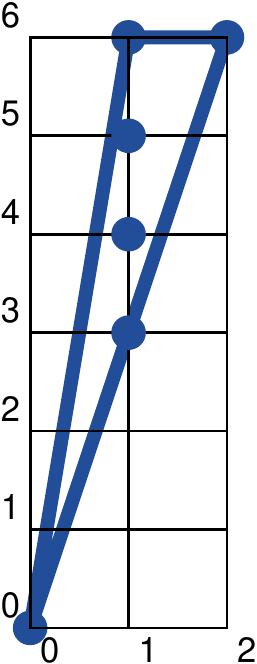}
\end{center}

\end{minipage}

By taking syzygies with support in $2P$, on gets a matrix representation of size $17 \times 34$.
The greatest common divisor of the $17$-minors of this matrix is the homogeneous implicit equation of the surface:
\begin{eqnarray*}
 & & 2809T_0^2T_1^4 + 124002T_1^6 - 5618T_0^3T_1^2T_2 + 66816T_0T_1^4T_2 +
2809T_0^4T_2^2\\
& &- 50580T_0^2T_1^2T_2^2  + 86976T_1^4 T_2^2 + 212T_0^3T_2^3  - 14210T_0T_1^2T_2^3  + 3078T_0^2 T_2^4 \\
& & + 13632T_1^2 T_2^4  + 116T_0T_2^5 + 841T_2^6  + 14045T_0^3 T_1^2 T_3 - 169849T_0T_1^4 T_3 \\
& & -14045T_0^4 T_2T_3 + 261327T_0^2 T_1^2 T_2T_3 - 468288T_1^4 T_2T_3 - 7208T_0^3 T_2^2 T_3 \\
& & + 157155T_0T_1^2 T_2^3 T_3 - 31098T_0^2 T_2^3 T_3 - 129215T_1^2 T_2^3 T_3 - 4528T_0T_2^4 T_3  \\
& & - 12673T_2^5 T_3 - 16695T_0^2 T_1^2 T_3^2  + 169600T_1^4 T_3^2  +
30740T_0^3 T_2T_3^2 \\
& & - 433384T_0T_1^2 T_2T_3^2 + 82434T_0^2 T_2^2 T_3^2  + 269745T_1^2 T_2^2 T_3^2  + 36696T_0T_2^3 T_3^2 \\
& &  + 63946T_2^4 T_3^2  + 2775T_0T_1^2 T_3^3  - 19470T_0^2 T_2T_3^4  + 177675T_1^2 T_2T_3^3  \\ 
& & - 85360T_0T_2^2 T_3^3  - 109490T_2^3 T_3^3  - 125T_1^2 T_3^4  + 2900T_0T_2T_3^4   \\
& & + 7325T_2^2 T_3^4  - 125T_2T_3^5 
\end{eqnarray*}
We can set $T_0=1$ to get the affine equation.

The map $g$ is computed with the following command:
\begin{verbatim}
g = teToricRationalMap f;
\end{verbatim}

The  matrix representation and the implicit equation are computed as follows:
\begin{verbatim}
M = representationMatrix (teToricRationalMap f,2);
implicitEq (L,2)
\end{verbatim}
Notice that the $2$ as second parameter in the computation of $\tt M$ is precisely the $2$ in the support
$2P$ of the syzygies. 
For a deeper understanding of the choice of the parameter $2$, see Appendix \ref{sec:catools}.

In the language of Section \ref{sec:tgtools}, the coordinate ring associated to $\Tc_P$ is $A=\kk[X_0,\ldots,X_5]/J_P$, 
where $J_P=(X_3^2-X_2X_4,X_2X_3-X_1X_4, X_2^2-X_1X_3,X_1^2-X_0X_5)$. The parametrization $g$ over $\Tc_P$ is given by 
$(g_0, g_1,g_2,g_3)=(2X_0+X_5, 2X_0+X_4,-3X_1+X_3, X_2+5X_5)$.
This matrix can be computed as the right-most map of the $\nu_0=2d =2$ strand 
of a graded complex as explained in Appendix~\ref{sec:catools}. 
The method fails over $\PP^2$  (i.e $P =$ the triangle with vertices $(0,0),(8,0),(0,8)$) due to the nature of the base 
locus. 
One can see this just by computing a  matrix representation and verifying that it is not full-ranked.

\subsection{Fewnomials with high degree.}\label{Ex:HighDeg}
This example contributes to show how the method works fine for high degree fewnomials involved in the parametrization. 

Consider the polynomials $(f_0,f_1,f_2,f_3)=(1 + st + s^{37}, s^7 + s^{47}, 
s^{37} + s^{59}, s^{61})$. Let
$f:\CC^2\dto \CC^3$ be the parametrization that maps 
$(s,t)\mapsto (f_1/f_0,f_2/f_0,f_3/f_0)(s,t)$. The implicit equation of the closure of the image of $f$
could be computed by 
eliminating the variables $(s,t)$ as follows (using general elimination procedures based on Gr\"obner bases in Macaulay2):

\begin{verbatim}
R = QQ[s,t, x, y, z, w]
f0 = 1 + s*t + s^37; f1 = s^7 + s^47; f2 = s^37 + s^59; f3 = s^61
L1 = x*f1 - y*f0; L2 = x*f2 - z*f0; L3 = x*f3 - w*f0 
eliminate ({s,t}, ideal(L1,L2,L3))
\end{verbatim}
In a 2014 standard desktop computer this routine does not end after one hour of  computation. We also
tried the well implemented {\sl eliminate} command in Singular \cite{Sing}, 
but with the same lack of answer after a couple of hours of
computation.

By homogenizing with an auxiliary variable $u$ we could try eliminate $(s,t,u)$ using 
Macaulay resultant methods, but we easily figure out that the homogenized forms $\overline{L}_1,\overline{L}_2,
\overline{L}_3$  
vanish identically over the point $(s,t,u)=(0:1:0)$. This implies in particular that the 
Macaulay resultant $\Res_{(s,t,u)}(\overline{L}_1,\overline{L}_2,\overline{L}_3)$ is identically zero.

Finally one can compute the implicit equation (and matrix representation) 
by implementing the implicitization techniques described in this article. 
A not very efficient (but efficient enough) routine in \cite{Bot10M2} 
gives the toric map $g$ in less than $2$ minutes and the desired 
 matrix representation $M$ in less than one more minute.

\subsection{Fewnomials with base points in the torus.}\label{Ex:BPtorus}
This example shows a case where classical resultants cannot be applied to compute the implicit equation, 
but the techniques in the paper can. Anyway, we recall that the aim of the matrix representations is to provide a better and more 
complete tool for representing a surface, and hence, the point presented with this example is just one extra advantage of the 
method.

Consider the following parametrization $(f_0,f_1,f_2,f_3)=(1-ts,-ts^{36}+1,-t(-s^{38}+t),s^{37}-t)$
given by four polynomials that define a parametrization $f:\CC^2\dto \CC^3$ that maps 
$(s,t)\mapsto (f_1/f_0,f_2/f_0,f_3/f_0)(s,t)$. The implicit equation cannot be computed by 
eliminating the variables $(s,t)$ with classical resultants, because the point $(1,1)$ is in the base locus.
However, this fact does not imply any problem in Algorithm~\ref{algo:mainalgo}.

With Algorithm~\ref{algo:Hirzebruch} with a rectangular $P=H_{38,2,0}$, it takes 0.058 seconds in a standard 2014 desktop 
computer to get a matrix representation $M$ of the closure of the image of $f$. The size of $M$
is  $152 \times 194$ and the gcd of its maximal minors has degree $110$.

\subsection{Comparison with and without embedding}\label{Ex:WithoutEmb}
While Algorithm~\ref{algo:mainalgo} holds with great generality, when dealing with
polynomials with rectangular support (which can be interpreted as bihomogeneous
polynomials), Algorithm~\ref{algo:Hirzebruch} provides a smaller matrix representation.

Consider the following four polynomials $f_0,\hdots,f_3$:

$f_0 = 3s_1^2s_2-2s_1s_2^2-s_1^2+s_1s_2-3s_1-s_2+4-s_2^2$,

$f_1 = 3s_1^2s_2-s_1^2-3s_1s_2-s_1+s_2+s_2^2+s_2^2+s_1^2s_2^2$,

$f_2 = 2s_1^2s_2^2-3s_1^2s_2-s_1^2+s_1s_2+3s_1-3s_2+2-s_2^2$, and 

$f_3 = 2s_1^2s_2^2-3s_1^2s_2-2s_1s_2^2+s_1^2+5s_1s_2-3s_1-3s_2+4-s_2^2$.

\noindent The Newton polytope $P=\Nc(f)$ is the rectangle $\{(x,y) \, : \, 0 \le x,y \le 2)\}$. 
If we apply Algorithm \ref{algo:mainalgo} (as we illustrated in Example \ref{example:IllustAlg}),
we obtain a  matrix representation $M$ of size $25\times 51$.

The associated toric variety $X_P$ can be identified with the $(2,2)$ Segre-Veronese embedding of  $\PP^1\times \PP^1$ 
in $\PP^{8}$ (see \cite{BD07,BDD08,Bot09}).

By means of Algorithm \ref{algo:Hirzebruch} we get a  matrix representation $M$ from a basis of linear syzygies
of bidegree $(2.2-1,2 -1)=(3,1)$. This matrix representation can be computed using
the algorithm developed in \cite{Bot10M2} and implemented in M2, as the matrix $M_\nu$ for bidegree 
$\nu=(3,1)$, and one obtains a square $8\times 8$-matrix.
Its determinant equals the implicit equation $F$:
{\footnotesize \begin{verbatim}
               8             7               6 2            5 3           4 4       
  F = 63569053X  - 159051916X X  + 175350068X X  - 82733240X X  + 2363584X X  +  ...
               0             0 1             0 1            0 1           0 1   
\end{verbatim}}
\noindent Notice that the matrix $M_{(3,1)}$ is considerably smaller than the $25\times 51$-matrix $M$
because instead of considering syzygies with support in the rectangle
$2P =\{(x,y) \, : \, 0 \le x,y \le 4 \}$, the syzygies are taken with support in the smaller rectangle
$\{(x,y) \, : \, 0 \le x \le 3, \, 0 \le y \le 1 \}$.



\renewcommand{\thesection}{A}
\setcounter{theorem}{0}

\section{Appendix: Commutative algebra tools}\label{sec:catools}

This appendix is devoted to highlight the tools of homological commutative algebra and algebraic
geometry that are
needed to justify the validity of Algorithms~\ref{algo:mainalgo} and~\ref{algo:Hirzebruch}, 
and to explain the choice of the support
of the syzygies in STEP 2 which define the matrix representation of the parametrized surface.

 Given $P$, the toric embedding $\rho:(\kk^*)^2 \rightarrow \PP^m$ in Section \ref{ssec:toricemb} provides a $\ZZ$-grading 
 in the coordinate ring $A=\kk[X_0,\hdots,X_m]/J_P$ in~\eqref{eq:A} of $\Tc_P$, which can be used to study the map $g$ in 
Diagram \ref{diagram2T} 
and its associated Rees and symmetric algebras, denoted by $\RIR$ and $\SIR$ respectively.
Notice also that the graded ring $A$ coincides with the affine semigroup ring of the lattice polytope $P$, 
which is {\sl Cohen-Macaulay} and normal because $P$ has dimension $2$.

The grading in $A$ plays a key role in the elimination process. The  matrix representation $M$ of Section 
\ref{ssec:general} 
depends on a choice of degree $\nu$, as was shown in Section \ref{ss:curves} for the 
case of curves. The reason why $\nu$ needs to be considered is rather technical, and a complete explanation involves sheaf 
cohomology.
From a more naive point of view, the implicit 
equation of the surface $\Sc := \overline{\im(g)}$ is written in the variables $T=(T_0,\dots,T_3)$ but 
depends on the algebraic relations among the polynomials $g_i$, which lie in $A$. Fixing a degree $\nu$ in $A$  can be thought as 
\textit{eliminating the variables of $A$, by hiding them in the monomial basis of the graded piece $A_\nu$}. 
In turn, recall that the variables $X=(X_0, \dots,X_m)$ in $A$ code monomials in the original $s$ variables, with exponents in the
lattice points in $P$.

More geometrically, consider the graph variety $\Gamma$ of $g$  where both group of variables $X$ and $T$
are involved. The  elimination process 
can be understood geometrically as projecting $\Gamma$ via $\pi_2$ in the following diagram 

\begin{equation*}\label{diagram2Tb}
\xymatrix{
\Gamma\subset\Tc_P \times \PP^3 \ar[dr]^{\pi_2} \ar[d]^{\pi_1} &  \\
\Tc_P \ar@{-->}[r]_g &\PP^3 \supset \overline{\im(g)} = \overline{\im(f)}.}
\end{equation*}

In the correspondence between 
subvarieties of $\Tc_P \times \PP^3$ and bigraded algebras, the inclusion of the graph 
$\Gamma\subset\Tc_P \times \PP^3$ 
corresponds 
to the surjection $A[T_0, T_1, T_2, T_3] \twoheadrightarrow \RIR$, the Rees algebra of the ideal $I$ generated by
$g_0, \dots,g_3$ over the 
coordinate ring $A$. The projection 
$\pi_2(\Gamma)$ corresponds to eliminating the variables $X_i$ of $\RIR$.
We denote by $I({\pi_2(\Gamma)})$ the defining ideal of $\pi_2(\Gamma)\subset  \PP^3 $.

How to eliminate the $X$ variables from $\RIR$ algebraically?
A standard procedure is to find a free graded presentation $F_1\nto{M}{\lto} F_2\to \RIR\to 0$ and a  degree $\nu$
(in the $X$ variables) such that the Fitting ideal
$\Fitt0(M_\nu)$ 
generated by the maximal minors of $M_\nu$ (in the graded strand $(F_1)_\nu\nto{M_\nu}{\lto} (F_2)_\nu\to \RIR_\nu\to 0$) 
computes $I({\pi_2(\Gamma)})$.
It happens that no universal way to compute such a free presentation is available, 
so the idea is to ``approximate'' $\RIR$ by the (hopefully) similar  
algebra $\SIR$ that admits such a universal resolution. These resolutions of the symmetric algebras
are known as approximation complexes, they were introduced in~\cite{HSV1,HSV2} 
and their application on elimination theory was 
done in~\cite{BuJo03, Buse1}.
The last map of the approximation complex is the following in our case:
\[
 Z_1[T_0,T_1,T_2,T_3]\nto{M'}{\lto}A[T_0,T_1,T_2,T_3]\to \SIR\to 0,
\]
where $Z_1=\llave{(a_0,a_1,a_2,a_3)\in A^4 : \sum a_ig_i=0}$ 
is the first module of syzygies of 
$g_0,g_1,g_2,g_3$ and $M'= [T_0\ T_1\ T_2\ T_3]^t$, that is,
\[
 M'\cdot (a_0,a_1,a_2,a_3):= \sum a_i T_i.
\]
The cokernel of $M'$ is 
$ \SIR =A[T_0,T_1,T_2,T_3]/ J$,
where 
\[J:=\{\sum a_i T_i : a_i\in A[T_0,T_1,T_2,T_3]\mbox{ and } \sum a_ig_i=0\}.\]	
We can recognize the origin of the linear forms $L_i$ in STEP 2 of our algorithms!

But there is a remaining question: which is the relation between $\RIR$ and $\SIR$? 
 Which variety does $\SIR$ define? Can we use  $\Fitt0(M'_\nu)$ to compute $\Ic_{\pi_2(\Gamma)}$ for 
some $\nu$?
The answer is that in case there are finitely many base points and for each base point $p$, the local $I_p$
is a local complete intersection, then
 $\RIR$ and $\SIR$ define the same scheme in $\Tc\times \AA^4$ (thus, in $\Tc\times \PP^3$). 
As $\RIR$ is $\mm$-torsion free, both algebras coincide module the $\mm$-torsion of $\SIR$, 
$\RIR\cong\SIR/H^0_\mm(\SIR)$, where $\mm$ is the maximal ideal generated by
$X_0, \dots,X_m$.  Thus if $I$ is a local complete intersection and $\nu$ is such that $H^0_\mm(\SIR)_\nu=0$, 
then $\RIR_\nu\cong\SIR_\nu$. This happens for $\nu\geq \nu_0:=2$ by Theorem~11 in~\cite{BDD08} (in fact,
that result also proves Theorem~\ref{th:refinement} as remarked before).
In particular, in this case, $\Fitt0(M'_\nu)$ computes $\Ic_{\pi_2(\Gamma)}$ for any $\nu\geq \nu_0$.

In fact, the condition of $I$ being a local complete intersection can be relaxed
to the condition of being locally an almost complete intersection.
(i.e.\ $I_p$ can be generated by $3$ elements, for any $p$ in the finite set $V(I)$).
In this case, $\dim (\SIR)=\dim (\RIR)$.
Since there is always a surjective map $\SIR\twoheadrightarrow\RIR$ 
then $ V(\SIR) = \Gamma \cup U$, where $U$ has the same dimension. In particular, $\pi_2(\Gamma) \cup \pi_2(U) \pi_2(V(\SIR)$.
For $\nu\geq \nu_0$, $\SIR$ is $\mm$-torsion free, and $\Fitt0(M'_\nu)$ 
computes $I({\pi_2(V(\SIR)})$ for any $\nu\geq \nu_0$. So, the gcd $H$ of the maximal minors of $M'\nu$ contains 
has the homogenization of the implicit equation $F$ as a factor.

In the bigraded case of Hirzebruch surfaces, in particular in the standard bigraded case, the basic
ideas are similar, but new technical details have to be managed in order to determine the bidegrees
for which the torsion of the symmetric algebra vanishes. We refer the reader to~\cite{Bot10} for
the details and proofs.

\section*{Acknowledgments} We are grateful to the organizers of the
European Marie Curie ITN SAGA network for the invitation to deliver 
a mini-course at the Fall School Shapes, Geometry, and Algebra
held on October 2010, at Kolympari, Greece, which was the starting point of this article.
We thank Laurent Bus\'e for interesting conversations.



\def\cprime{$'$}

\end{document}